\documentclass{amsart}
\newtheorem{Lem}{Lemma}
\newtheorem{Theo}[Lem]{Theorem}
\newtheorem{Prop}[Lem]{Proposition}
\newcommand{\intl}{\int\limits}

\usepackage{amssymb}

\begin{document}
\enlargethispage{5\baselineskip}
\title[Four prime cubes]{The density of integers representable as the sum of four prime cubes}
\author[C. Elsholtz]{Christian Elsholtz}
\address{Christian Elsholtz, Institut f\"ur Analysis und Zahlentheorie,
Technische  Universit\"at Graz, 
A-8010 Graz, Austria}
\author[J.-C. Schlage-Puchta]{Jan-Christoph Schlage-Puchta}
\address{Jan-Christoph Schlage-Puchta,
Mathematical institute,
University Rostock,
18957 Rostock,
Germany}

\begin{abstract}
The  set of integers which can be written as the sum of four prime cubes has
lower density at least $0.009664$. This improves earlier bounds of 
 $0.003125$ by Ren and $0.005776$ by Liu.
\end{abstract}
\maketitle
\thispagestyle{empty}
\section{Introduction and results}
The main result of this article is the following.
\begin{Theo}
\label{thm:4cubes}
The  set of integers which can be written as the sum of four prime cubes has
lower density $\omega \geq 0.009664$.
\end{Theo}
The same result had been obtained by Ren \cite{Ren4cubes} 
with a density of 0.003125, which had been improved by 
Liu \cite{Liu4cubes} to 0.005776. Note that apart from a set of density 0, an integer, which is representable as the sum of four prime cubes satisfies the congruence conditions $2|n$, $n\not\equiv\pm 1, \pm 3\pmod{9}$, $n\not\equiv\pm 1\pmod{7}$, thus the density of all integers of this form cannot be larger than $\frac{25}{126}=0.1984\ldots$.

The proof follows essentially the work of Ren \cite{Ren4cubes}, 
with considerably more precise estimates of an eight-dimensional
integral and the singular series. Define
$U=(N/(16+\delta))^{1/3}$, $V=U^{5/6}$, $L=\log N$, and denote by $r(n)$ the number of representations
of $n$ as $n=p_1^3+p_2^3+p_3^3+p_4^3$, where $U\leq p_1, p_2\leq 2U$, $V\leq
p_3, p_4\leq 2V$. (Note that the choice of the exponent $5/6$ goes back to Vaughan \cite{vaughan:1985}.)
Observe that 
\[ \frac{1+\varepsilon_1(\delta)}{8}N \leq n 
\leq  (1+\varepsilon_2(\delta))N, \]
where $\lim_{\delta \rightarrow 0} \varepsilon_i(\delta)=0$.
Let $I(\delta,N)=\left[ \frac{1+\varepsilon_1(\delta)}{8}N , 
(1+\varepsilon_2(\delta))N\right]$.
For very small $\delta$ one can approximately
think of $n \in [\frac{N}{8}, N]$.
Then the following holds:
\begin{Prop}
\label{prop:r^2 sum}
We have
\[
\sum_{n\in I(\delta,N)} r^2(n) \leq (C+o(1))UV^4L^{-8},
\]
where $C=100\, 552$,
for sufficiently small $\delta$.
\end{Prop}
Theorem~\ref{thm:4cubes} follows from Proposition~\ref{prop:r^2 sum} by a
simple application of Cauchy's inequality. Proposition~\ref{prop:r^2 sum} can
also be used to give estimates for integrals occurring in applications of the
circle method. We do not go into details here, but refer the reader to the paper \cite{Ren8cubes} by Liu and L\"u.

We have
\begin{eqnarray*}
\sum_{n \in I(\delta,N)} 
r^2(n) & = & \#\{p_1^3+p_2^3+p_3^3+p_4^2=p_5^3+p_6^3+p_7^3+p_8^3:\\
 &&\qquad\qquad U\leq p_1, p_2, p_5, p_6\leq 2U, V\leq p_3, p_4, p_7, p_8\leq 2V\}.\\
 & = & \#\{p_2, p_5, p_6\in[U, 2U], p_3, p_4, p_7, p_8 \in[V, 2V]:\\
 &&\qquad\qquad p_5^3+p_6^3+p_7^3+p_8^3 - p_2^3-p_5^3-p_6^3\mbox{ is the cube of a prime}\}
\end{eqnarray*}

To bound the quantity on the right we use an upper bound sieve, that is, we first consider the set $\mathcal{M}$ of all integers $m$, such that $p_5^3+p_6^3+p_7^3+p_8^3 - p_2^3-p_5^3-p_6^3 = m^3$, and then we sift out the possible prime values. To do so we need information on the distribution of the elements of $\mathcal{M}$ in residue classes.

Define $R(m)$ as the
number of solutions of the equation
$m^3+p_2^3+p_3^3+p_4^3=p_5^3+p_6^3+p_7^3+p_8^3$ with $U\leq p_2, p_5, p_6\leq
2U$, $V\leq p_3, p_4, p_7, p_8\leq 2V$.

We would expect that $R(m)$ is asymptotically equal to a local factor given by some singular series multiplied by a global factor given by an integral involving the density $\frac{1}{\log t}$ of the prime numbers. Unfortunately, we cannot prove this. In fact, we cannot even prove an asymptotic formula for the number of representations of an integer as the sum of 7 cubes. However, if we average over residue classes, then we can prove such a result. This is sufficient, as a linear sieve only requires information on the mean values $\underset{d|m}{\sum\limits_{U\leq m\leq 2U} } R(m)$. The archimedian part of the asymptotic formula for this sum is given by the integral
\[
J=\frac{1}{3}\int_\mathfrak{D}\frac{dv_1dv_2dv_3dv_4du_1du_2du_3du_4}
{v_1^{2/3}\cdots v_4^{2/3}u_1^{2/3}\cdots u_4^{2/3}
\log v_1\cdots\log v_4\log u_1\cdots \log u_4}
\]
where 
\begin{multline*}
\mathfrak{D}=\{(v_1, \ldots, v_4, u_1, \ldots, u_4):
V^3\leq v_1, \ldots, v_4\leq 8V^3, U^3\leq u_1, \ldots, u_4\leq 8U^3,\\
u_1+u_2+v_1+v_2=u_3+u_4+v_3+v_4\},
\end{multline*}
while the singular series is
\[
\mathfrak{S}_d = \sum_{q=1}^\infty T_d(q),
\]
where
\[
T_d(q)=\sum_{(a, q)=1} \frac{S(q, ad^3)C(q, a)^3
\overline{C(q, a)^4}}{q\varphi(q)^7},
\]
and
\[
S(q, a) = \sum_{m=1}^q e(\frac{am^3}{q}), \qquad C(q, a) = \sum_{(m,q)=1}
e(\frac{am^3}{q}).
\]
To summarize these considerations:
 the estimation of $\sum r^2(n)$ falls into four tasks: 
\begin{enumerate}
\item
estimate the distribution of $R$ on residue classes
\item
 estimate $J$
\item estimate $\mathfrak{S}_d$
\item
apply an upper bound sieve.
For this last step we use the weighted version of Iwaniecs's sieve due to Kawada and Wooley.
\end{enumerate}

\section{Step 1}
For the first part we use results by Ren \cite{Ren4cubes} and Liu \cite{Liu4cubes} as follows.

Define $E_d$ by means of the equation
\[
\underset{d|m}{\sum_{U\leq m\leq 2U} } R(m) = \frac{\mathfrak{S}_d}{d}J + E_d.
\] 
Then we have the following:

\begin{Lem}
\label{Lem:Ren circle}
\begin{enumerate}
\item $\mathfrak{S}_d$ is absolutely convergent and bounded independently of
  $d$. 
\item Let $\epsilon>0$ be fixed and let $\tau$ denote the divisor function. 
For all complex numbers $\eta_d$
  satisfying $|\eta_d|\leq\tau(d)$ and for all fixed $A$ the following estimate
holds:
\[
\sum_{d\leq N^{11/90-\epsilon}} \eta_d E_d \ll UV^4 L^{-A}.
\]
\end{enumerate}
\end{Lem}
The first part of the statement is contained in Ren's result \cite[Lemma~3.2]{Ren4cubes}, the second statement is due to Liu \cite[Lemma~3.2]{Liu4cubes}.

\section{Step 2}
We now compute $J$.
\begin{Lem}
We have $J\leq 440.62 \, UV^4L^{-7}$.
\end{Lem}
\begin{proof}
As each of the variables $u_1,u_2,u_3,u_4,v_1,v_2,v_3,v_4$ varies
over its range of integration in ${\mathfrak D}$, it
changes by a factor 8 at most. Hence its logarithm changes by
$\mathcal{O}(1)$, and therefore 
\[
\log v_i=(1+\mathcal{O}(\frac{1}{\log
  V}))\log V^3 = (1+o(1))\frac{5}{6}L
\]
   and
\[ 
\log u_i=(1+\mathcal{O}(\frac{1}{\log U}))\log U^3 = (1+o(1))L.
\]
We conclude that
\[
J=\frac{(1+o(1))6^4}{3\cdot 5^4L^8}\intl_{V^3}^{8V^3}\frac{dv_1}{v_1^{2/3}}\cdots\intl_{V^3}^{8V^3}\frac{dv_4}{v_4^{2/3}}\intl_{U^3}^{8U^3}\frac{du_1}{u_1^{2/3}}\intl_{U^3}^{8U^3}\frac{du_2}{u_2^{2/3}}\intl_{\max(U^3, x-8U^3)}^{\min(8U^3, x-U^3)}\frac{du_3}{u_3^{2/3}(x-u_3)^{2/3}},
\]
where $x-u_3=v_1+v_2-v_3-v_4+u_1+u_2-u_3=u_4$.
If we replace $x$ by $x'=x- \left(v_1+v_2-v_3-v_4\right)
=u_1+u_2-u_3$, the range of the innermost
integral changes by $\mathcal{O}(V^3)$. Since the integrand is
$\mathcal{O}(V^{-8} U^{-8})$, and the range of the first six integrals
is $\mathcal{O}(V^{12} U^6)$, this change in the range introduces an
error $\mathcal{O}(V^7 U^{-2})=\mathcal{O}(U^{23/6})$, 
which is of smaller order of magnitude than
$UV^4 L^{-7}\gg (U^{13/3-\varepsilon})$. The change of the innermost integrand is
\[
\ll \frac{|x-x'|}{u_3^{2/3}(x-u_3)^{5/3}} \ll \frac{|x-x'|}{U^7} \ll V^3U^{-7}.
\]
The total range of integration has measure $\ll V^{12} U^{9}$, and the
factors outside the innermost integral are
$\mathcal{O}(V^{-8}U^{-4})$, hence this change introduces an error of
magnitude $\mathcal{O}(V^7 U^{-2})=\mathcal{O}(U^{23/6})$, which is also negligible. Hence 
\begin{eqnarray*}
J  & = & \frac{(1+o(1))6^4}{3\cdot 5^4L^7}
\intl_{V^3}^{8V^3}\frac{dv_1}{v_1^{2/3}}\cdots \intl_{V^3}^{8V^3} \frac{dv_1}{v_4^{2/3}}
\intl_{U^3}^{8U^3}\frac{du_1}{u_1^{2/3}} \intl_{U^3}^{8U^3}
\frac{du_2}{u_2^{2/3}}\\
&&\qquad\qquad
\intl_{\max(U^3, u_1+u_2-8U^3)}^{\min(8 U^3, u_1+u_2-U^3)}
\frac{du_3}{u_3^{2/3} (u_1+u_2-u_3)^{2/3}}\\
 & = & \frac{18^4+o(1)}{3\cdot 5^4 L^7} V^4 \intl_{U^3}^{8U^3}
\frac{du_1}{u_1^{2/3}} \intl_{U^3}^{8U^3} \frac{du_2}{u_2^{2/3}}
\intl_{\max(U^3, u_1+u_2-8U^3)}^{\min(8\cdot U^3, u_1+u_2-U^3)}
\frac{du_3}{u_3^{2/3} (n+u_1+u_2-u_3)^{2/3}}\\
 & = & \frac{18^4+o(1)}{3\cdot 5^4 L^7} UV^4 \intl_{1}^{8}
\frac{dt_1}{t_1^{2/3}} \intl_{1}^{8} \frac{dt_2}{t_2^{2/3}}
\intl_{\max(1, t_1+t_2-8)}^{\min(8, t_1+t_2-1)}
\frac{dt_3}{t_3^{2/3}(t_1+t_2-t_3)^{2/3}}.
\end{eqnarray*}
The triple integral can be computed numerically to be in the range $[7.85, 7.87]$. We conclude that
\[
J\leq (1+o(1)) \frac{18^4}{3\cdot 5^4} 7.87 \frac{UV^4}{L^7}< 440.62 \frac{UV^4}{L^7},
\]
provided that $N$ is sufficiently large.
\end{proof}

The reader might wonder why we did not integrate numerically from the start,
however, giving a provable bound for an eight dimensional integral is quite a
delicate task.

\section{Steps 3 and 4}
We now apply the weighted version of Iwaniec's sieve due to Kawada and Wooley
(see \cite[Lemma~9.1]{KaWo}). We do not want to go into details here, as the
computations are the same as in the work of Ren \cite[Lemma~4.17]{Ren4cubes}. We obtain
\[
\sum_{n\in I(\delta,N)} r^2(n) \leq (1+o(1))e^\gamma J \mathfrak{S}_1W(z),
\]
where $\gamma=0.577\ldots$ is Euler's constant,
and 
\begin{eqnarray*}
W(z) & = & \prod_{p\leq N^{11/180-\epsilon}} \left(1-\frac{\omega(p)}{p}\right),\\
\mathfrak{S}_1 & = & (1+T_1(3)+T_1(9))\prod_{p\neq 3} (1+T_1(p)),\\
\omega(p) & = & \frac{1+T_p(p)}{1+T_1(p)}.
\end{eqnarray*}
The singular series $\mathfrak{S}_d$ mentioned in the introduction enters into the sieve,
its contribution can be seen in the structure of the product $W(z)$.
We have
\[
1-\frac{\omega(p)}{p} = \left(1-\frac{1}{p}\right)\left(1-\frac{T_p(p)-T_1(p)}{p-1}\right).
\]
Using Mertens' estimate
$\prod_{p<z}\left(1-\frac{1}{p}\right)\sim e^{-\gamma}\frac{1}{\log z}$ 
and the fact that
$\prod_{p\geq 11} \left(1-\frac{T_p(p)-T_1(p)}{p-1}\right)$
 converges we obtain for
$z=N^{11/180-\epsilon}$ 
\begin{eqnarray*}
\prod_{11\leq p\leq N^{11/180-\epsilon}}\left(1-\frac{\omega(p)}{p}\right) & = &
(1+o(1))\frac{e^{-\gamma}}{\log z}\cdot
2\cdot\frac{3}{2}\frac{5}{4}\frac{7}{6}\prod_{p\geq 11}\left(1-\frac{T_p(p)-T_1(p)}{p-1}\right)\\
 & \leq & \frac{40.197}{\log N}\prod_{p\geq 11}\left(1-\frac{T_p(p)-T_1(p)}{p-1}\right),
\end{eqnarray*}
provided that $\epsilon$ is sufficiently small.

If $p \mid a$, then $S(p,a)=p$.
If $p\nmid a$, then we have by Weil-estimates
$|S(p, a)|\leq 2\sqrt{p}$, and therefore $|C(p,
a)|\leq 2\sqrt{p}+1$. Hence
\[
|T_p(p)|\leq\left(\frac{2\sqrt{p}+1}{p-1}\right)^7,\quad
|T_1(p)|\leq\frac{2}{\sqrt{p}}\left(\frac{2\sqrt{p}+1}{p-1}\right)^7,
\]
and we see that the product over $p$ converges so fast that it can easily
be
evaluated numerically.

More precisely we have
\begin{multline*}
\prod_{p\geq P}\left(1-\frac{T_p(p)-T_1(p)}{p-1}\right) \leq \prod_{p\geq P}
\left(1+2\frac{(\sqrt{p}+2)(2\sqrt{p}+1)^7}{\sqrt{p}(p-1)^7}\right)\\
\leq \prod_{p\geq P}
\left(1+\frac{M}{p^{7/2}}\right)\leq
\left(\frac{\zeta(7/2)}{\zeta(7)}\prod_{p<P}\left(1+\frac{1}{p^{7/2}}\right)^{-1}\right)^M,
\end{multline*}
where
$M=2(1-\frac{1}{P})^{-7}(1+\frac{2}{\sqrt{P}})(2+\frac{1}{\sqrt{P}})^7$.

Using  $P=4000$ we obtain $M=279.551$, and write
\begin{eqnarray*}
\prod_{p\geq 11}\left(1-\frac{T_p(p)-T_1(p)}{p-1}\right) & \leq & \prod_{11\leq p\leq
  200}\left(1-\frac{T_p(p)-T_1(p)}{p-1}\right) \prod_{200\leq p\leq
  4000}\left(1+2\frac{(\sqrt{p}+2)(2\sqrt{p}+1)^7}{\sqrt{p}(p-1)^7}\right)\\
&&
\qquad\left(\frac{\zeta(7/2)}{\zeta(7)}\prod_{p<4000}\left(1+\frac{1}{p^{7/2}}\right)^{-1}\right)^M\\
& \leq & 1.029437(1+1.92\cdot 10^{-5})(1+4.54\cdot 10^{-11})\leq 1.02944.
\end{eqnarray*}
Plugging this value into the estimate above we obtain
\[
\prod_{11\leq p\leq N^{11/180-\epsilon}}\left(1-\frac{\omega(p)}{p}\right)
\leq
\frac{41.38}{\log N}.
\]

Similarly we can estimate $\mathfrak{S}_1$ as
\begin{eqnarray*}
\mathfrak{S}_1 & \leq & (1+T_1(3)+T_1(9))\prod_{5\leq p\leq
500}(1+T_1(p))\prod_{500\leq p\leq
4000}\left(1+\frac{2(2\sqrt{p}+1)^7}{\sqrt{p}(p-1)^7}\right)\\
&&\qquad
\left(\frac{\zeta(4)}{\zeta(8)}\prod_{p<4000}\left(1+\frac{1}{p^4}\right)^{-1}\right)^{270.982}\\
& \leq & 3.0963(1+1.19\cdot 10^{-7})(1+1.64\cdot 10^{-10})\leq 3.0964.
\end{eqnarray*}

\section{Putting these results together}
Putting these results together we obtain
\[
\sum_{n\in I(\delta,N)} r^2(n) \leq (1+o(1)) e^\gamma 440.62 \cdot 41.3794
 \cdot 3.0964 \cdot UV^4L^{-8} \leq 100\, 552 UV^4 L^{-8}
\]
and the proof of Proposition~\ref{prop:r^2 sum} is complete.

To deduce Theorem~\ref{thm:4cubes} we apply the Cauchy-Schwarz inequality in
much the same way as in the proof of Romanov's theorem. We have
\[
\sum_{n\in I(\delta,N)} r(n) = (\pi(2U)-\pi(U))^2(\pi(2V)-\pi(V))^2 \sim
\frac{3^4\cdot6^2}{5^2}U^2V^2 L^{-4},
\]
thus
\begin{multline*}
\underset{r(n)\geq 1}{\sum_{n\in I(\delta,N)} }1 \geq \Big(\sum_{n\in I(\delta,N)}
r(n)\Big)^2\Big(\sum_{n\in I(\delta,N)}  r^2(n)\Big)^{-1}\\ 
\geq \frac{3^8\cdot6^4}{5^4\cdot  100\, 552} U^3 =(1+O(\delta)) \frac{3^{12}}{5^4\cdot  100\, 552} N\geq 0.00845638 N,
\end{multline*}

for sufficiently small $\delta$.
Patching intervals of the form $[N/8, N]$ together Theorem~\ref{thm:4cubes} follows by multiplying the last density with $8/7$.

The authors would like to thank the referee for useful comments on the manuscript.

\ \\
\textbf{Key words and MSC:}\\
Goldbach-Waring problem, applications of sieve methods, sums of cubes\\
Primary: 11P32 Goldbach-type theorems; other additive questions involving primes\\
Secondary: 11P05 Waring's problem and variants\\
11N36 Applications of sieve methods 
\end{document}